\documentclass[11pt,a4paper]{amsart}

\usepackage[left=2.7cm,right=2.7cm,top=3.5cm,bottom=3cm]{geometry}

\usepackage[english]{babel}
\usepackage[T1]{fontenc}
\usepackage[latin1]{inputenc}
\usepackage{amssymb,amsmath,amsthm}
\usepackage{tikz-cd}
\usepackage[hidelinks]{hyperref}
\usepackage[shortlabels]{enumitem}
\usepackage{stmaryrd}
\usepackage{mdwlist}
\usepackage{url}
\usepackage{mathrsfs}
\usepackage[normalem]{ulem}

\newtheorem{definition}{Definition}[section]

\newtheorem{proposition}{Proposition}[section]
\newtheorem{lemma}{Lemma}[section]

\newtheorem{conjecture}{Conjecture}[section]

\newtheorem*{theorem*}{Theorem}
\newtheorem*{assumption*}{Assumption}
\newtheorem*{conjecture*}{Conjecture}

\theoremstyle{remark}
\newtheorem{remark}{Remark}[section]
\newtheorem*{remark*}{Remark}

\numberwithin{equation}{section}

\DeclareMathOperator{\GL}{GL}
\DeclareMathOperator{\SL}{SL}
\DeclareMathOperator{\Ann}{Ann}
\DeclareMathOperator{\depth}{depth}
\DeclareMathOperator{\id}{id}
\DeclareMathOperator{\Spec}{Spec}
\DeclareMathOperator{\diag}{diag}

\newcommand{\utau}{\uline{\tau}}


\begin{document}

\title{Matrix-valued Hilbert modular forms}
\author{Enrico Da Ronche}
\thanks{The author is partially supported by the GNSAGA group of INdAM}

\begin{abstract}
In this paper we generalize the notion of logarithmic vector-valued modular form in order to give a general definition of matrix-valued Hilbert modular forms. We prove that they admit unique polynomial Fourier expansions and we build examples in some particular cases.
\end{abstract}

\address{Dipartimento di Matematica, Universit\`a di Genova, Via Dodecaneso 35, 16146 Genova, Italy}
\email{enrico.daronche@edu.unige.it}

\maketitle

\tableofcontents

\section{Introduction}

In \cite{KM1} and \cite{KM2}, Knopp and Mason introduced the notion of vector-valued modular forms. In particular, given $k \in \mathbb{Q}$ and a complex representation $\rho: \SL_2(\mathbb{Z}) \rightarrow \GL_r(\mathbb{C})$, they defined a vector-valued modular form to be a tuple $(f_1, \dots, f_r)$ where $f_i: \mathcal{H} \rightarrow \mathbb{C}$ is an holomorphic function for any $i$ and such that
\begin{itemize}
\item $(c \tau + d)^{-k}\begin{pmatrix} f_1(\gamma \tau) \\ \vdots \\ f_r(\gamma \tau) \end{pmatrix} = \rho(\gamma) \begin{pmatrix} f_1(\tau) \\ \vdots \\ f_r(\tau)\end{pmatrix}$ for any $\gamma= \begin{pmatrix} a & b \\ c & d \end{pmatrix} \in \SL_2(\mathbb{Z})$;
\item $f_i$ admits a convergent holomorphic Fourier expansion
$$f_i(\tau)=\displaystyle \sum_{n \geq 0} a_{i,n} q^{n/N_i}$$
for some integers $N_i \geq 1$ and for any $i$.
\end{itemize}
Then, in \cite{KM}, they generalized it by introducing the unrestricted vector-valued (and matrix-valued) modular forms, i.e., vector-valued modular forms that satisfy only the first property listed above. They proved that they admit certain polynomial Fourier expansions, i.e, any component function has a convergent $q$-expansion of the form
$$f_i(\tau)=\displaystyle \sum_{u \in U} \displaystyle \sum_{t \in S_u} \tau^t f_{i,u,t}(\tau)$$
where
\begin{itemize}
\item $U \subset \mathbb{C}$ is finite;
\item $S_u$ is a finite subset of $\mathbb{N}$ for any $u \in U$;
\item $f_{i,u,t}(\tau)=\displaystyle \sum_{n \in \mathbb{Z}} a_{i,u,t,n} e^{2 \pi i (n+u)\tau}$ with $a_{i,u,t,n} \in \mathbb{C}$.
\end{itemize}
This led to the theory of logarithmic vector-valued modular forms. Knopp and Mason also built Poincar\'e series that are concrete examples of these functions and studied some properties of their Fourier coefficients.

The purpose of this paper is to generalize the work done by Knopp and Mason in \cite{KM} to the case of Hilbert modular forms, i.e., modular forms over a totally real field $F$. Our definition of matrix-valued Hilbert modular function follows naturally from the one of unrestricted vector-valued modular form in the classical case while the theory of polynomial Fourier expansions is more complicated. In particular, we are going to consider matrices of functions
$$\begin{pmatrix} g_{1,1}(\utau) & \dots & g_{1,c}(\utau) \\
\vdots & \ddots & \vdots \\
g_{r,1}(\utau) & \dots & g_{r,c}(\utau) \end{pmatrix}$$
where each $g_{i,j}:\mathcal{H}^n \rightarrow \mathbb{C}$ is an holomorphic function ($n=[F:\mathbb{Q}]$) and which satisfies a formal transformation law with respect to a complex representation of $H$, a finite index subgroup of $\SL_2(\mathcal{O}_F)$. We prove that all the component functions admit a unique polynomial Fourier expansion
$$g_{i,j}(\utau)=\displaystyle \sum_{u \in U}\sum_{\uline{t} \in S_u}  \utau^{\uline{t}} g_{i,j,u,\uline{t}}(\utau) $$
where
\begin{itemize}
\item $U \subset \mathbb{C}^n$ is finite;
\item $S_u$ is a finite subset of $ \mathbb{N}^n$ for any $u \in U$;
\item $\utau^{\uline{t}}:=\prod_{i=1}^n \tau_i^{t_i}$;
\item $g_{i,j,u,\uline{t}}(\utau)=\sum_{v \in \Lambda_H^*} a_{i,j,u,\uline{t},v} e^{2 \pi i ((v+u) \cdot \utau)}$ with $a_{i,j,u,\uline{t},v} \in \mathbb{C}$.
\end{itemize}
Starting from this result we define matrix-valued Hilbert modular forms and the spaces where they lie. In order to prove that our definitions are meaningful we also build examples of these functions in case of unitary representations, in particular we define Poincar\'e series
$$G(\utau)= \displaystyle \sum_{M \in \Lambda \textbackslash \Gamma_F} \rho(M)^{-1} \diag( \dots, e^{2 \pi i (\uline{\nu_j} + \uline{\mu_j}) \cdot M \utau} \cdot \mathbb{I}_{m_j}, \dots) P(M \utau)^{-1} J_{\uuline{k}}(M,\utau)^{-1}$$
(the notation is introduced in the next sections) and we prove that they are matrix-valued Hilbert modular forms.
This paper is mainly devoted to give the definitions and to study the theory of polynomial Fourier expansions but, of course, it can be considered just as a starting point for a more deep research on this family of modular forms.

\subsection{Structure of the paper}

In order to ease the reading of the paper we give a brief description of its structure. In \autoref{sec:2} we introduce the notation and the objects (coming from the theory of Hilbert modular forms) we are going to work with and we define matrix-valued Hilbert modular functions. In \autoref{sec:3} we use some linear algebra and well-known results about Fourier expansions of Hilbert modular forms to prove that these functions admits a unique polynomial Fourier expansion. The \autoref{sec:4} is devoted to the definition of matrix-valued Hilbert modular forms, to the spaces generated by these functions and to the algebraic structures they can be endowed with. In \autoref{sec:5} we build concrete examples under some hypotheses, i.e., we define Poincar\'e series and we prove that they are absolutely convergent and, in particular, matrix-valued Hilbert modular forms. Finally, \autoref{sec:A} is a remark about a well-known algebraic property of spaces of vector-valued modular forms that, maybe, can be generalized to the totally real setting.

\subsection{Notation} 
We denote by $\id_n$ the trivial complex representation of dimension $n$ and by $\mathbb{I}_n$ the identity matrix of dimension $n$.

\subsection*{Acknowledgements}
I would like to thank Ignacio Mu\~{n}oz Jim\'enez for several helpful discussions.

\section{Matrix-valued Hilbert modular functions} \label{sec:2}

We introduce the notion of matrix-valued Hilbert modular forms. Let $F$ be a totally real field and let $n:=[F:\mathbb{Q}]$ and $\Gamma_F:=\SL_2(\mathcal{O}_F)$. Let us denote by
$$\iota_F:F \hookrightarrow \mathbb{R}^n, \: \iota_F(f):=(\sigma_1(f), \dots,\sigma_n(f))$$
the obvious injective map, where $\sigma_1,\dots,\sigma_n$ are the embeddings of $F$ in $\mathbb{R}$. Let $r,c \geq 1$ be integers and denote by $M_{r,c}(\mathbb{C})$ the space of $r \times c$ complex-valued matrices. We also fix a finite index subgroup $H \leq \Gamma_F$ and a complex representation $\rho:H \rightarrow \GL_r(\mathbb{C})$. For any function $G:\mathcal{H}^n \rightarrow M_{r,c}(\mathbb{C})$, any $\gamma = \begin{pmatrix} a & b \\ c & d \end{pmatrix} \in H$ and any matrix $\uuline{k} \in M_{c,n}(\mathbb{Z})$ we define the automorphy factor as
$$J_{\uuline{k}}(\gamma,\utau):=\begin{pmatrix} j_{\uline{k_1}}(\gamma,\utau) & 0 & \dots &0 \\
0 & \ddots & \ddots & \vdots\\
\vdots & \ddots & \ddots & 0\\ 0 & \cdots & 0 & j_{\uline{k_c}}(\gamma, \utau)  \end{pmatrix}$$
where $\uline{k_i}$ is the $i$-th row of $\uuline{k}$ and
$$j_{\uline{k_i}}(\gamma,\uline{\tau}):=\displaystyle \prod_{j=1}^n (\sigma_j(c) \tau_j+ \sigma_j(d))^{k_{i,j}}.$$
We also set
$$(G|_{\uuline{k}} \gamma )(\uline{\tau}):=G(\gamma \cdot \uline{\tau}) J_{\uuline{k}}(\gamma,\uline{\tau})^{-1}. $$

\begin{definition}
Let $\uuline{k} \in M_{c,n}(\mathbb{Z})$ and let $\rho:H \rightarrow \GL_r(\mathbb{C})$ be a complex representation. A \textbf{matrix-valued Hilbert modular function} of weight $\uuline{k}$ with respect to $\rho$ is a holomorphic function
$$G:\mathcal{H}^n \rightarrow M_{r,c}(\mathbb{C})$$
such that $(G|_{\uuline{k}} \gamma)(\uline{\tau})=\rho(\gamma) G(\uline{\tau})$ for any $\gamma \in H$.

\end{definition}

We will use the following notation:
$$G(\uline{\tau})=\begin{pmatrix} g_{1,1}(\utau) & \dots & g_{1,c}(\utau) \\
\vdots & \ddots & \vdots \\
g_{r,1}(\utau) & \dots & g_{r,c}(\utau) \end{pmatrix}.$$

We say that $\uuline{k}=\begin{pmatrix} \uline{k_1} \\ \vdots \\ \uline{k_c} \end{pmatrix}$ is the matrix of weights of $G$. 

\section{Polynomial Fourier expansions} \label{sec:3}

In order to define matrix-valued Hilbert modular forms we need to define polynomial Fourier expansions of matrix-valued Hilbert modular functions. First of all, we give a definition and a preliminary result coming from linear algebra.

\begin{definition}
A family of matrices $A_1, \dots ,A_n \in M_r(\mathbb{C})$ is simultaneously block-triangularizable with same dimensions (SBTSD) if there exists a basis of $\mathbb{C}^r$ such that any matrix is in the form
$$A_i=\begin{pmatrix} A_{i,1} & 0 &0 \\ 0 & \ddots & 0 \\ 0 & 0 & A_{i,k} \end{pmatrix}$$
where $k \geq 1$ is an integer and $A_{i,j}$ is an upper-triangular square block with equal elements on the diagonal and such that $\dim A_{i,j}=\dim A_{i',j}$ for any $1 \leq i,i' \leq n$.
\end{definition}

\begin{proposition} \label{cm}
Any family $A_1, \dots ,A_n \in M_r(\mathbb{C})$ of commuting matrices, i.e. $A_iA_j=A_jA_i$ for any $i,j$, is SBTSD.
\end{proposition}

\begin{proof}
First of all let us fix the decomposition of $\mathbb{C}^r$ given by the Jordan canonical form of $A_1$, i.e.
$$\mathbb{C}^r=\displaystyle \bigoplus_{\lambda \in \Spec(A_1)} \ker (A_1-\lambda)^r.$$
Then, we notice that for any $\lambda \in \Spec(A_1)$ and any $i$ we have that $V_{1,\lambda}:=\ker (A_1-\lambda)^r$ is stable under $A_i$ since
$$(A_1 - \lambda)^r A_i=A_i (A_1- \lambda)^r.$$
Then we can put $A_2 \big|_{V_{1,\lambda}}$ in Jordan canonical form for any $\lambda \in \Spec(A_1)$ and by iterating this process a finite number of times we can get a decomposition
$$\mathbb{C}^r=\displaystyle \bigoplus_{j=1}^t W_j$$
where, for any $i$ and any $j$ there exists $\lambda \in \Spec(A_i)$ such that $(A_i \big|_{W_j}-\lambda)^r=0$. Finally, we can conclude since the family $\{ A_i \big|_{W_j}\}_{j=1, \dots, t}$ consists of commuting matrices and so, in particular, they are simultaneously triangularizable.
\end{proof}

Let $H$ be a finite index subgroup of $\Gamma_F$. We set
$$S_H:=\left\{ a \in \mathcal{O}_F: \begin{pmatrix} 1 & a \\ 0 & 1 \end{pmatrix}  \in H \right\}$$
and we denote by $\Lambda_H:=\iota_F(S_H)$ the translation lattice of $H$. We also denote by
$$\Lambda_H^*:=\left\{ v \in \mathbb{R}^n: v \cdot u \in \mathbb{Z} \: \text{for any} \: u \in \Lambda_H \right\}$$
the dual lattice of $\Lambda_H$. We have $S_H \cong \mathbb{Z}^n$ and so we can fix a $\mathbb{Z}$-basis $a_1,\dots,a_n$ of $S_H$. We denote by $v_i:=\iota_F(a_i)$ and by $A_i:=\rho \left( T_i \right)$ where $T_i:=\begin{pmatrix} 1 & a_i \\ 0 & 1 \end{pmatrix}$. Obviously, $v_1,\dots,v_n$ is a $\mathbb{Z}$-basis of $\Lambda_H$ and we denote by $M$ the matrix whose columns are $v_1,\dots,v_n$.

\begin{lemma} \label{pe}
Let $f:\mathcal{H}^n \rightarrow \mathbb{C}$ be an holomorphic function such that
$$f(\utau+v_i)=\lambda_i f(\utau)$$
for any $i=1,\dots,n$, where the $\lambda_i=e^{2 \pi i \mu_i}$ are non-zero complex numbers. Then it admits a Fourier expansion of the form
$$f(\utau)=\displaystyle \sum_{v \in \Lambda_H^*} a_v e^{2 \pi i ((v+\mu \cdot M^{-1}) \cdot \utau)}$$
where $\mu:=(\mu_1,\dots,\mu_n)$.
\end{lemma}

\begin{proof}
We set
$$g(\utau):=f(\utau) \cdot \exp \left( -2 \pi i(\mu_1,\dots,\mu_n) \cdot M^{-1} \cdot \utau \right)$$
and we notice that
\begin{flalign*}
&& g(\uline{\tau}+v_t) &=f(\uline{\tau}+v_t) \cdot \exp \left(-2 \pi i (\mu_1,...,\mu_n) \cdot M^{-1} \cdot (\uline{\tau}+v_t) \right) &\\
&& &=\lambda_t f(\uline{\tau}) \cdot \exp \left( -2 \pi i (\mu_1,...,\mu_n) \cdot M^{-1} \cdot \uline{\tau} - 2 \pi i \mu_t \right) &\\
&& &=\lambda_t f(\uline{\tau}) \cdot \exp \left( -2 \pi i(\mu_1,...,\mu_n) \cdot M^{-1} \cdot \uline{\tau} \right) \cdot \lambda_t^{-1} =g(\uline{\tau}).
\end{flalign*}
Then by \cite[Lemma 4.1]{F} $g$ admits a $q$-expansion of the form
$$g(\uline{\tau})= \displaystyle \sum_{v \in \Lambda_H^*} a_v e^{2 \pi i (v \cdot \uline{\tau})}$$
and so we can write $f$ in the form
$$f(\uline{\tau})=\displaystyle \sum_{v \in \Lambda_H^*} a_v e^{2 \pi i (v \cdot \uline{\tau})} \cdot e^{2 \pi i \mu \cdot M^{-1} \cdot \uline{\tau}}=\displaystyle \sum_{v \in \Lambda_H^*} a_v e^{2 \pi i ((v+\mu \cdot M^{-1})\cdot \uline{\tau})}$$
as stated.
\end{proof}

Let $W_j:=\langle g_{1,j}(\utau), \dots , g_{r,j}(\utau) \rangle_\mathbb{C}$ for any $j=1, \dots ,c$. For any $i$ and any $j$ we set a linear operator $N_{i,j}$ on $W_j$ by setting:
$$N_{i,j} (c_1 g_{1,j}(\utau) + \dots + c_r g_{r,j}(\utau)):=(c_1, \dots , c_r) \cdot A_i \cdot \begin{pmatrix} g_{1,j}(\utau) \\ \vdots \\ g_{r,j}(\utau) \end{pmatrix}$$
where $c_1, \dots, c_r \in \mathbb{C}$. Since $G$ is a matrix-valued Hilbert modular function, we have
$$\begin{pmatrix} g_{1,j}(\utau + v_i) \\ \vdots \\ g_{r,j}(\utau + v_i) \end{pmatrix}=A_i \begin{pmatrix} g_{1,j}(\utau) \\ \vdots \\ g_{r,j}(\utau) \end{pmatrix}$$
and, thanks to this, the operators $N_{i,j}$ are well-defined. By abuse of notation, we denote by $N_{i,j}$ also the matrix of the associated operator for a fixed basis of $W_j$. It is immediate to see that they commute with each other and they are invertible. By Proposition \ref{cm} the family $\{N_{i,j} \}_{i=1, \dots, n}$ is SBTSD. In particular, we can fix basis $\{l_{1,j}(\utau), \dots , l_{s(j),j}(\utau) \}$ of $W_j$ such that the matrices are in the desired form. We also notice that
$$\begin{pmatrix}l_{1,j}(\utau+v_i) \\ \vdots \\ l_{s(j),j}(\utau + v_i) \end{pmatrix} = N_{i,j}^T \begin{pmatrix} l_{1,j}(\utau) \\ \vdots \\ l_{s(j),j}(\utau) \end{pmatrix}.$$

\begin{proposition}
Let $j=1, \dots, n$. There exists a matrix $M_j(\uline{\tau}) \in M_{s(j),s(j)}(\mathbb{C}[\utau])$ such that
$$\begin{pmatrix} l_{1,j}(\utau) \\ \vdots \\ l_{s(j),j}(\utau) \end{pmatrix} = M_j(\utau) \begin{pmatrix} h_{1,j}(\utau) \\ \vdots \\ h_{s(j),j}(\utau) \end{pmatrix}$$
where, for any $i$, there exists $u \in \mathbb{C}^n$ such that
$$h_{i,j}(\utau)=\displaystyle \sum_{v \in \Lambda_H^*} a_{i,j,v} e^{2 \pi i ((v+u) \cdot \utau)}.$$
\end{proposition}

\begin{proof}
We just need to prove the statement for triangular blocks of the matrices, so we assume that
$$N_i=\begin{pmatrix} 
\lambda_i & a^i_{1,2} & \dots & a^i_{1,s} \\
0 & \ddots & \ddots & \vdots\\
\vdots & \ddots & \ddots  & a^i_{s-1,s}\\
0 & \dots &0 & \lambda_i
\end{pmatrix}$$
for any $i=1,\dots,n$ in the basis $\{l_1(\uline{\tau}),\dots,l_s(\uline{\tau}) \}$ (we omit the index $j$ from the notation).
Since $N_i$ is invertible for any $i$ we can set $H_i:= N_i^T /\lambda_i$ and $S_i:=H_i^{-1}$ for any $i=1,\dots,n$. We fix a logarithm for any matrix $S_i$ which is lower-triangular with entries equal to $0$ on the diagonal (it exists since the $S_i$ are invertible and lower-triangular with entries equal to $1$ on the diagonal) and we set
$$P(\uline{\tau}):=\exp \left( \displaystyle \sum_{i=1}^{n} (M^{-1} \cdot \uline{\tau})_i \log(S_i) \right)$$
and
$$\begin{pmatrix} h_1(\uline{\tau}) \\ \vdots \\ h_s(\uline{\tau})\end{pmatrix}:= P(\uline{\tau}) \cdot \begin{pmatrix} l_1(\uline{\tau}) \\ \vdots \\ l_s(\uline{\tau})\end{pmatrix}.$$
We have $P(\utau) \in \GL_s(\mathbb{C}[\utau])$ since it is the exponential of a lower-triangular matrix in $M_{s,s}(\mathbb{C}[\utau])$ with entries equal to $0$ on the diagonal (in particular, it has entries equal to $1$ on the diagonal).
We observe that
$$P(\uline{\tau}+v_t)=\exp \left( \displaystyle \sum_{i=1}^{n} (M^{-1} \cdot \uline{\tau})_i \log(S_i) + \log(S_t) \right)=P(\uline{\tau}) \cdot \exp (\log(S_t))=P(\uline{\tau}) \cdot S_t$$
and we compute
\begin{flalign*}
&& \begin{pmatrix} h_1(\uline{\tau}+v_t) \\ \vdots \\ h_s(\uline{\tau}+v_t)\end{pmatrix} &= P(\uline{\tau}+v_t) \cdot \begin{pmatrix} l_1(\uline{\tau}+v_t) \\ \vdots \\ l_s(\uline{\tau}+v_t)\end{pmatrix} = P(\uline{\tau}) \cdot S_t \cdot N_t^T \cdot \begin{pmatrix} l_1(\uline{\tau}) \\ \vdots \\ l_s(\uline{\tau})\end{pmatrix} &\\
&& &= P(\uline{\tau}) \cdot H_t^{-1} \cdot \lambda_t H_t \cdot \begin{pmatrix} l_1(\uline{\tau}) \\ \vdots \\ l_s(\uline{\tau})\end{pmatrix}= \lambda_t \begin{pmatrix} h_1(\uline{\tau}) \\ \vdots \\ h_s(\uline{\tau})\end{pmatrix}.
\end{flalign*}
By Lemma \ref{pe} we know that $h_i$ admits a Fourier expansion of the form
$$h_i(\uline{\tau})=\displaystyle \sum_{v \in \Lambda_H^*} a_v e^{2 \pi i ((v+ \mu \cdot M^{-1}) \cdot \utau)}$$
and we can conclude by observing that
$$\begin{pmatrix} l_1(\uline{\tau}) \\ \vdots \\ l_s(\uline{\tau}) \end{pmatrix}=P(\uline{\tau})^{-1} \cdot \begin{pmatrix} h_1(\uline{\tau}) \\ \vdots \\ h_s(\uline{\tau}) \end{pmatrix}.$$
\end{proof}

Finally, we have that for any integer $1 \leq j \leq c$ there exists $U \subset \mathbb{C}^n$ finite such that for any integer $1 \leq i \leq r$ the function $g_{i,j}(\utau)$ can be written in the form
$$g_{i,j}(\utau)=\displaystyle \sum_{u \in U}\sum_{\uline{t} \in S_u}  \utau^{\uline{t}} g_{i,j,u,\uline{t}}(\utau) $$
where
\begin{itemize}
\item $S_u$ is a finite subset of $ \mathbb{N}^n$ for any $u \in U$;
\item $\utau^{\uline{t}}:=\prod_{i=1}^n \tau_i^{t_i}$;
\item $g_{i,j,u,\uline{t}}(\utau)=\sum_{v \in \Lambda_H^*} a_{i,j,u,\uline{t},v} e^{2 \pi i ((v+u) \cdot \utau)}$ with $a_{i,j,u,\uline{t},v} \in \mathbb{C}$.
\end{itemize}
We call it the \textbf{polynomial Fourier expansion} of $g_{i,j}$.

\begin{proposition}
The polynomial Fourier expansion is unique.
\end{proposition}

\begin{proof}
In order to ease the reading, we fix $i$ and $j$ and we remove them from the notation. We assume
$$g(\utau)=\displaystyle \sum_{u \in U} \displaystyle \sum_{\uline{t} \in S_u} \utau^{\uline{t}} g_{u,\uline{t}}(\utau) \equiv 0$$
and we set
$$g_u(\utau):=\displaystyle \sum_{\uline{t} \in S_u} \utau^{\uline{t}} g_{u,\uline{t}}(\utau).$$
If $g_{u,\uline{t}}(\utau) \equiv 0$ we have
$$\sum_{v \in \Lambda_H^*}a_{u,\uline{t},v}e^{2 \pi i v \cdot \utau} \cdot e^{2 \pi i u \cdot \utau} \equiv 0 \Rightarrow \sum_{v \in \Lambda_H^*}a_{u,\uline{t},v}e^{2 \pi i v \cdot \utau} \equiv 0$$
and so $a_{u,\uline{t},v}=0$ for any $v \in \Lambda_H^*$ by \cite[Lemma 4.1]{F}. So, we just need to prove that $g_{u,\uline{t}}(\utau) \equiv 0$ for any $u,\uline{t}$. We fix a $\mathbb{Z}$-basis $v_1, \dots , v_n$ of $\Lambda_H$ and we set $\lambda_{i,u}:=e^{2 \pi i u \cdot v_i}$. In particular, we have $g_{u,\uline{t}}(\utau+v_i)=\lambda_{i,u} g_{u,\uline{t}}(\utau)$. We say that $\uline{t_1} < \uline{t_2}$ if $t_{1,i} < t_{2,i}$ for any $i=1, \dots,n$ and we define the following weak derivatives:
$$d_ig_u(\utau):=\frac{g_u(\utau+v_i)}{\lambda_{i,u}}-g_u(\utau),$$
$$d_{i,u}g(\utau):=\frac{g(\utau+v_i)}{\lambda_{i,u}}-g(\utau).$$
We notice that if there are $u \neq u'$ in $U$ with $\lambda_{i,u}=\lambda_{i,u'}$ for any $i$ we can consider $u=u'$ since we have $v_i \cdot (u-u') \in \mathbb{Z}$ for any $i$, so $u-u' \in \Lambda_H^*$ and, in particular, $\{v+u\}_{v \in \Lambda_H^*}=\{v+u'\}_{v \in \Lambda_H^*}$.
We fix $u_0 \in U$ and $\uline{t_0} \in S_{u_0}$ such that there is not $\uline{t} \in S_{u_0}$ with $\uline{t_0} < \uline{t}$ and we assume $|U| \geq 2$. Let $u' \in U \setminus \{ u_0\}$ and fix $i$ such that $\lambda_{i,u'} \neq \lambda_{i,u_0}$. We have
\begin{flalign*}
&& d_{i,u'}g(\utau) &= \frac{g(\utau + v_i)}{\lambda_{i,u'}} -g(\utau) &\\
&& &= \displaystyle \sum_{u \in U} \left( \frac{g_u(\utau + v_i)}{\lambda_{i,u'}} - g_u(\utau) \right) &\\
&& &=d_ig_{u'}(\utau)+ \displaystyle \sum_{u \in U \setminus \{u' \}} \left( \frac{\lambda_{i,u}}{\lambda_{i,u'}} \cdot \frac{g_u(\utau+v_i)}{\lambda_{i,u}} - \frac{\lambda_{i,u}}{\lambda_{i,u'}} \cdot g_u(\utau) + \left( \frac{\lambda_{i,u}}{\lambda_{i,u'}}-1\right)g_u(\utau) \right) &\\
&& &= d_ig_{u'}(\utau)+ \displaystyle \sum_{u \in U \setminus \{u'\}} \left( \frac{\lambda_{i,u}}{\lambda_{i,u'}}d_ig_u(\utau) + \left( \frac{\lambda_{i,u}}{\lambda_{i,u'}}-1\right)g_u(\utau) \right) &
\end{flalign*}
and we set
$$h(\utau):=\left(\frac{\lambda_{i,u_0}}{\lambda_{i,u'}}-1 \right)^{-1} d_{i,u'}g(\utau)=\displaystyle \sum_{u \in U}h_u(\utau)=\sum_{u \in U} \sum_{\uline{t} \in S'_u} \utau^{\uline{t}} h_{u,\uline{t}}(\utau).$$
Thanks to the previous computation we have:
\begin{itemize}
\item $h(\utau) \equiv 0$;
\item $h_{u'}(\utau)=\left(\frac{\lambda_{i,u_0}}{\lambda_{i,u'}}-1 \right)^{-1} d_i g_{u'}(\utau)$;
\item $h_{u_0,\uline{t_0}}(\utau)=g_{u_0,\uline{t_0}}(\utau)$;
\item there is not $\uline{t} \in S'_{u_0}$ with $\uline{t_0} < \uline{t}$.
\end{itemize}
Now, we can compute $d_{i,u'}h(\utau)$ and, since $d_i^t g_{u'}=0$ for some $t \in \mathbb{N}_{>0}$, if we iterate the process for a finite number of times we can assume that
$$g(\utau)=\displaystyle \sum_{u \in U \setminus \{u'\}} g_u(\utau)$$ 
and, proceding in the same way for any $u' \in U \setminus \{u_0\}$, we can also assume that
$$g(\utau)=g_{u_0}(\utau)=\sum_{\uline{t} \in S_{u_0}} \utau^{\uline{t}} g_{\uline{t},u_0}(\utau).$$
We notice that there exist $h_1, \dots, h_n \in \mathbb{N}_{>0}$ such that $h_i e_i \in \Lambda_H$ for any $i$. So, we set $\lambda'_{i,u}:=e^{2 \pi i u_i h_i}$ and we define the following weak derivative:
$$d'_{i,u}g(\utau):=\frac{g(\utau + h_i e_i)}{\lambda'_{i,u}}-g(\utau).$$
Since there is not $\uline{t} \in S_{u_0}$ with $\uline{t_0} < \uline{t}$, it is immediate to see that
$$\displaystyle \prod_{i=1}^n (d'_{i,u_0})^{t_{0,i}}g(\utau)=m \cdot g_{u_0,\uline{t_0}}(\utau) \equiv 0$$
with $m \in \mathbb{N}_{>0}$. Finally, $g_{u_0,\uline{t_0}} \equiv 0$ and by iterating the reasoning the statement is proved.
\end{proof}

\begin{definition}
We say that $g_{i,j}(\utau)$ is holomorphic at $\infty$ if
$$a_{u,\uline{t},v} \neq 0 \Rightarrow v_k+ \Re(u_k) \geq 0$$
for any complex coefficient in the Fourier expansion and for any $k$. We say that $G(\utau)$ is holomorphic at $\infty$ if $g_{i,j}(\utau)$ is holomorphic at $\infty$ for any $i,j$.
\end{definition}

\section{Spaces of matrix-valued Hilbert modular forms} \label{sec:4}

\begin{definition}
A \textbf{matrix-valued Hilbert modular form} is a matrix-valued Hilbert modular function $G: \mathcal{H}^n \rightarrow M_{r,c}(\mathbb{C})$ which is also holomorphic at $\infty$.
\end{definition}

Given $\alpha_1,\alpha_2 \in \mathbb{C}$ and two matrix-valued Hilbert modular forms $G_1,G_2: \mathcal{H}^n \rightarrow M_{r,c}(\mathbb{C})$ of weight $\uuline{k}$ with respect to a complex representation $\rho$ it is easy to see that $\alpha_1G_1+\alpha_2G_2$ is again a matrix-valued Hilbert modular form of the same weight and with respect to the same representation. We denote by $M^F_{\uuline{k}}(\rho)$ the complex vector space of matrix-valued Hilbert modular forms of weight $\uuline{k}$ with respect to $\rho$ and we set
$$M^F_c(\rho):=\displaystyle \bigoplus_{\uuline{k} \in M_{c,n}(\mathbb{Z})} M^F_{\uuline{k}}(\rho).$$

When $c=1$ we say that $G$ is a vector-valued Hilbert modular form. In particular, we notice that 
$$M^F_{\uuline{k}}(\rho)= \bigoplus_{i=1}^c M_{\uline{k_i}}^F(\rho), \:\:\:\:\:\:\:\: M^F_c(\rho)=(M^F_1(\rho))^c.$$
We can also endow $M^F_c(\rho)$ with the structure of a graded module over $M_1^F(\id_1)$ (it is the classical ring of Hilbert modular forms over $F$). In particular, if $g \in M^F_{\uline{\alpha}}(\id_1)$ and $G \in M^F_{\uuline{\beta}}(\rho)$ we set
$$(g \cdot G)(\uline{\tau}):=g (\utau) \cdot G(\utau)$$
and it is easy to see that $g \cdot G \in M^F_{\uuline{k}}(\rho)$ where $\uuline{k}=\begin{pmatrix} \uline{\alpha} \\ \dots \\ \uline{\alpha} \end{pmatrix} + \uuline{\beta}$.

\section{Poincar\'e series} \label{sec:5}

We build Poincar\'e series in order to prove that matrix-valued Hilbert modular forms exist. Let $F$ be a totally real field of degree $n$ over $\mathbb{Q}$, let $\rho: \Gamma_F \rightarrow \GL_r(\mathbb{C})$ be a complex representation and let $\uuline{k} \in M_{c,n}(\mathbb{Z})$ with $k_{i,j}>2$ for any $i,j$. We keep the notations for $\Lambda$, $T_i$ and $A_i$ as in \autoref{sec:3}. We fix a basis of $\mathbb{C}^r$ such that $A_1, \dots, A_n$ are in the form described in Proposition \ref{cm} with $t$ blocks and we denote by $\lambda_{i,j}$ the eigenvalue attached to the $j$-th block of $A_i$. We also set $S_i$ to be the inverse of the matrix obtained from $A_i$ by dividing each block by $\lambda_{i,j}$. We set
$$\uuline{\mu}:=\begin{pmatrix} \uline{\mu_1} \\ \dots \\ \uline{\mu_t} \end{pmatrix}=\begin{pmatrix} \mu_{1,1} & \dots & \mu_{1,n} \\ \vdots & & \vdots \\ \mu_{t,1} & \dots & \mu_{t,n} \end{pmatrix} \in M_{t,n}(\mathbb{C})$$
where $e^{2 \pi i \mu_{i,j}}=\lambda_{i,j}$. We also fix
$$\uuline{\nu}:=\begin{pmatrix} \uline{\nu_1} \\ \dots \\ \uline{\nu_t} \end{pmatrix}=\begin{pmatrix} \nu_{1,1} & \dots & \nu_{1,n} \\ \vdots & & \vdots \\ \nu_{t,1} & \dots & \nu_{t,n} \end{pmatrix} \in M_{t,n}(\mathbb{R})$$
where $\uline{\nu_i} \in \Lambda^*$ for any $i$. Finally, we set
$$P(\utau):=\exp \left(\displaystyle \sum_{i=1}^n \uline{\tau_i} \log (S_i) \right)$$
and we define the Poincar\'e series as follows:
$$G(\utau)= \displaystyle \sum_{M \in \Lambda \textbackslash \Gamma_F} \rho(M)^{-1} \diag( \dots, e^{2 \pi i (\uline{\nu_j} + \uline{\mu_j}) \cdot M \utau} \cdot \mathbb{I}_{m_j}, \dots) P(M \utau)^{-1} J_{\uuline{k}}(M,\utau)^{-1}$$
where $m_j$ is the dimension of the $j$-th block.
First of all we notice that the series is formally well-defined. In particular, for any $M \in \Gamma_F$ we have
\begin{flalign*}
&& &\rho(T_i M)^{-1} \diag( \dots, e^{2 \pi i (\uline{\nu_j} + \uline{\mu_j}) \cdot T_i M \utau} \cdot \mathbb{I}_{m_j}, \dots) P(T_i M \utau)^{-1} J_{\uuline{k}}(T_i M,\utau)^{-1}= &\\
&&  &\rho(M)^{-1} \rho(T_i)^{-1} \diag( \dots, e^{2 \pi i (\uline{\nu_j} + \uline{\mu_j}) \cdot (M \utau + v_i)} \cdot \mathbb{I}_{m_j}, \dots)  P(M \utau +v_i)^{-1} J_{\uuline{k}}(M,\utau)^{-1}= &\\
&& &\rho(M)^{-1} \rho(T_i)^{-1} \diag( \dots, e^{2 \pi i (\uline{\nu_j} + \uline{\mu_j}) \cdot M \utau} \cdot \mathbb{I}_{m_j}, \dots) \diag(\dots , \lambda_{i,j} \cdot \mathbb{I}_{m_j} , \dots) S_i^{-1} P(M \utau)^{-1} J_{\uuline{k}}(M,\utau)^{-1}= &\\
&& & \rho(M)^{-1} \diag( \dots, e^{2 \pi i (\uline{\nu_j} + \uline{\mu_j}) \cdot M \utau} \cdot \mathbb{I}_{m_j}, \dots) P(M \utau)^{-1} J_{\uuline{k}}(M,\utau)^{-1}.
\end{flalign*}

Furthermore, it is easy to see that $G$ satisfies the transformation law of modular functions:
\begin{flalign*}
&& G(\gamma \cdot \utau) J_{\uuline{k}}(\gamma,\utau)^{-1} &= \displaystyle \sum_{ M \in \Lambda \textbackslash \Gamma_F} \rho(M)^{-1} \diag( \dots, e^{2 \pi i (\uline{\nu_j} + \uline{\mu_j}) \cdot M \gamma \utau} \cdot \mathbb{I}_{m_j}, \dots) P(M \gamma \utau)^{-1} J_{\uuline{k}}(M,\gamma \utau)^{-1} J_{\uuline{k}}(\gamma,\utau)^{-1} &\\
&& &=\rho(\gamma) \displaystyle \sum_{ M \in \Lambda \textbackslash \Gamma_F} \rho(M \gamma)^{-1} \diag( \dots, e^{2 \pi i (\uline{\nu_j} + \uline{\mu_j}) \cdot M \gamma \utau} \cdot \mathbb{I}_{m_j}, \dots) P(M \gamma \utau)^{-1} J_{\uuline{k}}(M \gamma, \utau)^{-1}&\\
&& &=\rho(\gamma) G(\utau).
\end{flalign*}

Given a matrix $A=(a_{i,j})_{i,j} \in \GL_r(\mathbb{C})$ we define its norm by $\lVert A \rVert:=\max_{i,j} |a_{i,j}|$. We recall that $\lVert AB \rVert \leq \lVert A \rVert \lVert B \rVert$ for any $B \in \GL_r(\mathbb{C})$.

From now on, we assume that, considering the fixed basis of $\mathbb{C}^r$, the following conditions hold:
\begin{itemize}
\item the matrices in $\rho(\Lambda)$ are diagonal and their eigenvalues have absolute value equal to one (in particular, $\uuline{\mu} \in M_{t,n}(\mathbb{R})$;
\item $\lVert \rho(S) \rVert =1$ where $S=\begin{pmatrix} 0 & -1 \\ .1 & 0 \end{pmatrix}$.
\end{itemize}
We also assume $\uuline{\nu} \in M_{t,n}(\mathbb{R}_{>0})$.

We observe that the previous assumptions are satisfied, for example, if $\rho$ is a unitary representation (unitary commuting matrices are simultaneously unitarily diagonalizable).
We are going to prove that, under these assumptions, $G$ is a matrix-valued Hilbert modular form in $M_{\uuline{k}}^F(\rho)$.

\begin{remark}
Of course, we expect that $G$ or a proper variant is a matrix-valued Hilbert modular form even without the assumptions we made on $\rho$ but, at the moment, it is not clear to the author how to overcome some problems related to convergence. It seems that results in \cite{S} can be useful in this direction.
\end{remark}

\begin{proposition}
The series defining $G$ is absolutely convergent on $\mathcal{H}^n$.
\end{proposition}

\begin{proof}
Thanks to the main theorem of \cite{V}, we know that $\Gamma_F$ is generated by the family $\{S,T_1, \dots , T_n\}$. Since we also know, by assumption, that $\lVert \rho(S) \rVert = 1$ and $\lVert A_i\rVert = 1$ we have that $\lVert \rho(M) \rVert \leq 1$ for any $M \in \Gamma_F$. We also notice that
$$P(\utau)=\exp \left(\displaystyle \sum_{i=1}^n \uline{\tau_i} \log (\mathbb{I}_r) \right)=\mathbb{I}_r.$$
Then we have
\begin{flalign*}
&& \lVert G(\utau) \rVert & \leq \displaystyle \sum_{M \in \Lambda \textbackslash \Gamma_F} \left\lVert \diag( \dots, e^{2 \pi i (\uline{\nu_j} + \uline{\mu_j}) \cdot M \utau} \cdot \mathbb{I}_{m_j}, \dots)   J_{\uuline{k}}(M,\utau)^{-1} \right\rVert &\\
&& & \leq \displaystyle \sum_{j=1}^t \displaystyle \sum_{i=1}^c \displaystyle \sum_{M \in \Lambda \textbackslash \Gamma_F} \left| e^{2 \pi i \uline{\nu_j} \cdot M \utau} j_{\uline{k_i}}(M,\utau)^{-1}\right| < \infty
\end{flalign*}
where the last sum is convergent by \cite[$\mathsection 1.13$]{GA}.
\end{proof}

In order to finish, we notice that any component function of $G$ is of the form
$$g(\utau)= \displaystyle \sum_{M \in \Lambda \textbackslash \Gamma_F} a_M e^{2 \pi i (\uline{\nu_j} + \uline{\mu_j}) \cdot M \utau} j_{\uline{k_i}}(M, \utau)^{-1}$$
where $a_M \in \mathbb{C}$ depends only on $M$ and $|a_M| \leq 1$. Thanks to this we have that $G$ is holomorphic since the Poincar\'e series defined in \cite[$\mathsection 1.13$]{GA} are absolutely-uniformly convergent on compact subsets of $\mathcal{H}^n$. We fix $\utau=\lambda e_k$ for $\lambda \in \mathbb{R}_{>0}$ and $k=1, \dots ,n$ and we have
$$\lim_{ \lambda \to +\infty} \left| \displaystyle \sum_{M \in \Lambda \textbackslash \Gamma_F} a_M e^{2 \pi i (\uline{\nu_j} + \uline{\mu_j}) \cdot M \utau} j_{\uline{k_i}}(M, \utau)^{-1} \right| \leq \lim_{\lambda \to +\infty} \sum_{M \in \Lambda \textbackslash \Gamma_F} \left| e^{2 \pi i (\uline{\nu_j}) \cdot M \utau} j_{\uline{k_i}}(M,\utau)^{-1} \right|=0$$
(see, for example, \cite[page 53]{GA}) and we can conclude that
$$\lim_{ \lambda \to + \infty} g(\utau)=0.$$

\appendix

\section{Commutative algebra of $M_c^F(\rho)$} \label{sec:A}

\begin{proposition} \label{cmds}
Let $R$ be a Noetherian commutative ring and let $M_1,M_2$ be finitely generated Cohen-Macaulay modules over $R$. Then $M_1 \oplus M_2$ is Cohen-Macaulay if and only if $\dim_R(M_1)=\dim_R(M_2)$.
\end{proposition}

\begin{proof}
We just need to prove the statement when $R$ is a local Noetherian commutative ring. We notice that
\begin{flalign*}
&& \dim_R(M_1 \oplus M_2) &=\dim_R(R/\Ann_R(M_1 \oplus M_2)) &\\
&& &=\dim_R(R/(\Ann_R(M_1) \cap \Ann_R(M_2))) &\\
&& &=\max \{\dim_R(R/\Ann_R(M_1)), \dim_R(R/\Ann_R(M_2)) \} &\\
&& &= \max \{\dim_R(M_1),\dim_R(M_2) \}.
\end{flalign*}
The third equality follows from the basic fact that any prime ideal of $R$ which contains $\Ann_R(M_1) \cap \Ann_R(M_2)$ must contain either $\Ann_R(M_1)$ or $\Ann_R(M_2)$.
On the other hand we have
\begin{flalign*}
\depth_R(M_1 \oplus M_2)=\min \{ \depth_R(M_1), \depth_R(M_2)\}.
\end{flalign*}
Finally $M_1 \oplus M_2$ is Cohen-Macaulay if and only if $\dim_R(M_1 \oplus M_2)=\depth_R(M_1 \oplus M_2)$, i.e. if and only if
$$\max \{ \dim_R(M_1),\dim_R(M_2) \} = \min \{ \depth_R(M_1),\depth_R(M_2) \}.$$
Since $M_1$ and $M_2$ are Cohen-Macaulay it holds if and only if
$$\max \{ \dim_R(M_1),\dim_R(M_2) \} = \min \{ \dim_R(M_1),\dim_R(M_2) \}$$
and the statement follows.
\end{proof}

Let $r,c \geq 1$ be integers, let $F$ be a totally real field and let $\rho:\Gamma_F \rightarrow \GL_r(\mathbb{C})$ be a complex representation.

\begin{proposition}
$M^F_1(\id_1)$ is a Noetherian ring.
\end{proposition}

\begin{proof}
We know that $M^F_1(\id_1)$ is a finitely generated algebra over $\mathbb{C}$ (see \cite[page 139]{TV}). Then it is Noetherian.
\end{proof}

It is natural to formulate the following conjecture since it is known that it holds if $F=\mathbb{Q}$.

\begin{conjecture} \label{ccm}
$M^F_c(\rho)$ is a finitely generated Cohen-Macaulay module over $M^F_1(\id_1)$.
\end{conjecture}

\begin{proposition}
Conjecture \ref{ccm} is true when $F=\mathbb{Q}$.
\end{proposition}

\begin{proof}
By \cite[Theorem 1.2]{G} and \cite[Theorem 1.3]{G} we know that $M^\mathbb{Q}_1(\rho)$ is a finitely generated Cohen-Macaulay module over $M^\mathbb{Q}_1(\id_1)$. By \cite[Lemma 3.2]{G} we also know that the Krull dimension of all these modules is $2$. Then we can conclude by Proposition \ref{cmds}.
\end{proof}

\bibliographystyle{amsplain}
\bibliography{Database}

\end{document}